\documentclass[11pt]{amsart}
\usepackage{a4wide}
\usepackage[dvips]{epsfig}
\usepackage{amscd}
\usepackage{amssymb}
\usepackage{amsthm}
\usepackage{amsmath}
\usepackage{latexsym}

\theoremstyle{plain}
\newtheorem{thm}{Theorem}[section]
\theoremstyle{plain}
\newtheorem{lem}[thm]{Lemma}
\newtheorem{prop}[thm]{Proposition}

\theoremstyle{definition}

{%
\setcounter{enumi}{0}

\begin{enumerate}}%
{\end{enumerate} }
{%
\setcounter{enumi}{0}

\begin{enumerate}}%
{\end{enumerate} }

\numberwithin{equation}{section}
 \allowdisplaybreaks

\title[Hedging spread Options]
 {On the sensitivity analysis of spread options using Malliavin calculus}

\date{\today}

\begin{document}
\author{ Farai Julius Mhlanga }

\address{Department of Mathematics and Applied Mathematics,
University of Limpopo, Private bag X1106, Sovenga, 0727, South Africa}

\email{farai.mhlanga@ul.ac.za}

\author{ Shadrack Makwena Kgomo }

\address{Department of Mathematics and Applied Mathematics
, University of Limpopo, Private bag X1106, Sovenga, 0727, South Africa}

\email{shadrack.kgomo@ul.ac.za}






\begin{abstract}
In this paper we derive tractable formulae for price sensitivities of two-dimensional spread options using Malliavin calculus. In particular, we consider spread options with asset dynamics driven by geometric Brownian motion and stochastic volatility models. Unlike the fast Fourier transform approach, the Malliavin calculus approach does not require the joint characteristic function of underlying assets to be known and is applicable to spread options with discontinuous payoff functions. The results obtained reveal that the Malliavin calculus approach gives the price sensitivities in terms of the expectation of spread option payoff functional multiplied with some random variables (Malliavin weights) which are independent of the payoff functional. This is consistent with results in Fourni$\acute{\text{e}}$ et al. \cite{Fo}. The results also show the flexibility of Mallavin calculus approach when applied to spread options.

\end{abstract}

\maketitle

{\em Classifications}: 60H07, 60G05, 60J60
\section{Introduction}
A spread option is an option whose payoff is based on the difference (i.e. the spread) between two or more underlying assets. Spread options are common in several markets such as the fixed-income, currency, commodity, and energy markets \cite{Alf, Bel}. For instance, in the fixed income markets, a popular product in the United States of America is the Note Against Bond (NOB) spread in which a yield curve is created between the 30-year bond futures contract (long position) and the 10-year US Treasury note  futures (short position). In the commodity markets, spread options are based on the difference between the prices of the same commodity at two different locations or at two different points in time, as well as between the prices of different grades of the same commodity \cite{Carmona}. In the energy markets, crack spread options and spark spread options are prevalent. The crack spread is based on the differential between the price of crude oil and refined petroleum products. The spread represents the refinement margin made by the oil refinery by ``cracking" the crude oil into a refined petroleum product \cite{Eydeland}. The spark spread refers to differences between the price of electricity and the price of fuel, and are are widely used by power plant operators to optimize their revenue streams \cite{Hurd}. Spread options are sometimes traded on exchanges, but most often as  over-the-counter transactions.

Spread options are popular because they are designed to mitigate adverse movements between two market variables \cite{Bel, Errera}. For instance, the crack spread are used by refineries to hedge against price fluctuations, mitigate risk, or secure a profit margin on the production output \cite{Bel}. The spark spread represents the margin of the power plant, which takes fuel to run its generator to produce electricity. However, spread options are complex contracts as they take two underlying assets as the reference point in their payoff and they are harder to value compared to plain vanilla call and puts due to the 2-dimensionality \cite{Bel, Gui}.  Nonetheless, these options represent interesting alternatives for those seeking coverage of positions in several assets.

The sensitivity analysis is carried over parameters appearing in the model for the price dynamics. Price sensitivities are derivatives (in the classical sense) of the price of spread options with respect to parameters of the model. For example, \emph{Delta}, denoted by $\Delta_i$, $i=1,2,$ is the derivative of the price of spread option with respect to the initial price of the underlying asset. \emph{Gamma}, denoted by $\Gamma_i$, $i=1,2,$ is the second derivative of the price of spread option with respect to the initial price of the underlying assets. \emph{Vega}, denoted by $\mathcal V_i$, $i=1,2,$ is the derivative of the price of spread option with respect to the volatility of the underlying assets \cite{Mh}. For spread options, two values are obtained for each which are associated with the price of the underlying assets under consideration.

The derivation of price sensitivities for spread options is a challenging task due to lack of closed formulae for their prices. Analytical methods applicable to log-normal models that involve linear approximations of the nonlinear exercise boundary has been used. Kirk \cite{Kik} presented an analytical approximation which is a generalization of Margrabe's formula \cite{Mag} for an exchange option, that is, a spread option with zero-strike price. Kirk's formula performs well in practice. Carmona and Durrleman \cite{Carmona} and later Li et al. \cite{Li} derive a number of lower and upper bounds for the spread option price that combine to produce accurate analytical approximation formulas in log-normal asset models. These results were used to approximate values of price sensitivities via direct differentiation. Alfeus and Schl$\ddot{\text{o}}$gl \cite{Alf} calculated the change in the spread option value with respect to change in volatility parameters of each asset assuming that the model (joint) characteristic function is known, on price obtained using fast Fourier transform approach. Hurd and Zhou \cite{Hurd} and Dempster and Hong \cite{Dem} also explored the use of fast Fourier transform in three different models for spread options on two stocks, namely the geometric Brownian motion, stochastic volatility model and the variance gamma model. A formula for calculating the \emph{vega} (sensitivity to volatility) was presented using direct differentiation.

Over the past two decades, a fairly large and growing literature have been developed around the computation of price sensitivities using Malliavin calculus. Fourni$\acute{\text{e}}$ et al. \cite{Fo} introduce the application of Malliavin calculus to the computation of price sensitivities on markets driven by the Brownian motion only. Their work was further extended by several authors. El-Khatib and Privault \cite{El} use Malliavin calculus on Poisson space to derive price sensitivities in a market driven by the Poisson processes. Davis and Johansson \cite{Joh} utilise Malliavin calculus to calculate price sensitivities in a jump diffusion setting assuming a separability condition. Petrou \cite{Pet} derived price sensitivities using Malliavin calculus for markets driven by square integrable L\'evy processes. Both Malliavin calculus and Fourier transform were used by Benth et al. \cite{Ben} to compute price sensitivities within a jump-diffusion framework. Mhlanga \cite{Mh1} makes use of Malliavin calculus to compute price sensitivities where small jumps for L\'evy processes were approximated by a Brownian motion. Yilmaz \cite{Yil} uses Malliavin calculus approach to compute price sensitivities for underlying asset and interest rate involving stochastic volatility and stochastic interest rate models, respectively. Kawai and Kohatsu-Higa \cite{Kaw} use  Malliavin calculus to obtain expressions for price sensitivities for an asset price dynamics model defined with time-changed Brownian motion. In all these references the Malliavin calculus approach was not applied to spread options. The Malliavin calculus has a derivation property and its adjoint coincides with the It$\hat{\text{o}}$ stochastic integral  on adapted processes, which provides a natural way to make explicit computations of weights.

The purpose of this paper is to derive closed-form formulae for price sensitivities of two-dimensional spread options using Malliavin calculus. The paper focuses on asset dynamics that are driven by geometric Brownian motion and stochastic volatility models. The general L\'evy models are left for further studies. The fast Fourier transform, which has been applied in most papers dealing with spread options, requires that the joint characteristic function of the underlying assets be known in advance. In practice, the joint characteristic function may not be available, for example, in spread options of Asian-type, necessitating the need for alternative approaches. The Malliavin calculus adopted in the present paper does not require the joint characteristic function to be known and is applicable to spread options with discontinuous payoff functions. This demonstrates the flexibility of Malliavin calculus approach. The contribution of this paper is to provide tractable formulae for price sensitivities of spread options in the context of Malliavin calculus. In particular, we provide formulae for the \emph{Delta}, \emph{Gamma} and \emph{Vega} with respect to each of the underlying asset prices. In passing, we generalize the calculation of price sensitivities of spread options proposed in Carmona and Durrleman \cite{Carmona}, Alfeus and Schl$\ddot{\text{o}}$gl \cite{Alf},  Hurd and Zhou \cite{Hurd}, and Li et al. \cite{Li}. We also discuss the localised Malliavin calculus which improves the Malliavin calculus. The localised Malliavin calculus approach use Malliavin calculus approach only around the point of discontinuity, and direct methods outside. Our results reveal that the Malliavin calculus approach gives the price sensitivities in terms of the expectation of spread option payoff functional multiplied with some Malliavin weights which are independent of the payoff functional. This is consistent with results in Fourni$\acute{\text{e}}$ et al. \cite{Fo}. Our results gain importance in view of the application of Malliavin calculus for the computation of price sensitivities of spread options in mathematical finance.

The paper is structured as follows. In Section 2, we describe the spread option pricing problem. Section 3 is devoted to a discussion on Malliavin calculus. We provide an overview of necessary tools needed in our proofs. Section 4 specify the market setting considered in this paper. In Section 5, we present theoretical formulae for computing the price sensitivities. Examples are provided in Section 6. The spread option with stochastic volatility is presented in Section 7. Section 8 deals with the localization of Malliavin calculus when computing price sensitives of spread options. Section 9 concludes this paper.

\section{Mathematical setup}
For a fixed $T$, we consider a probability space $(\Omega,{\mathcal F},\{\mathcal F\}_{0\leq t\leq T},P)$ defined in the usual sense. Consider a spread option of European call type between two stock price processes $S_1=(S_1(t))_{0\leq t\leq T}$ and $S_2=(S_2(t))_{0\leq t\leq T}$ with maturity time $T$ and exercise price $K$. Its payoff at time $T$ is given by
\[ (S_2(T)-S_1(T)-K)^+\]
where $(\cdot)^+=\max\{\cdot,0\}$. At maturity, if the spread $S_2(T)-S_1(T)$ is greater than the exercise price $K$, the option holder exercises the option and gains the difference between the spread and the strike price. If the spread is less than $0$, the option holder does not exercise the option, and the payoff is $0$. The price of the spread option $u$ at time $t=0$ is expressed by the risk-neutral expectation
\begin{equation}\label{fa2}
u=\mathbb E[e^{-rT}\left(S_2(T)-S_1(T)-K\right)^+]
\end{equation}
where $r$ is the risk-free interest rate, which is here assumed to be constant.  \\
Define $\Phi(S_1(T),S_2(T)):=\left(S_2(T)-S_1(T)-K\right)^+$. Then (\ref{fa2}) can be expressed as follows
\begin{equation}\label{fa3}
u=\mathbb E[e^{-rT}\Phi(S_1(T),S_2(T))].
\end{equation}
Equation (\ref{fa3}) shows the price of the spread option with payoff function being a function of terminal values of two assets $S_1(T)$ and $S_2(T)$.
\section{A Primer on Malliavin calculus}
In this section we recall some of the basic properties of Malliavin calculus as highlighted in Fourni$\acute{\text{e}}$ et al. \cite{Fo} and Mhlanga \cite{Mh}. We refer to Nualart \cite{N} for a detailed exposition on Malliavin calculus.\\
For $h(\cdot)\in H=L^2([0,T],\mathbb R^d)$, denote by $W(h)$ the Wiener stochastic integral $\int_0^Th(t)dW_t$. Let ${\mathcal S}$ denote the class of random variables of the form
\[
F=f(W(h_1),...,W(h_n))
\]
where $f\in C_p^{\infty}(\mathbb R^n)$, $(h_1,...,h_n)\in H^n$ and $n\geq 1$.  For $F\in {\mathcal S}$, we define the Malliavin derivative $DF=\left(D_tF\right)_{t\in[0,T]}$ of $F$ as the $H$-valued random variable given by
\begin{equation}
\label{a}
 DF = \sum_{i=1}^n \frac{\partial f}{\partial x_i}(W(h_1),...,W(h_n))h_i(t).
\end{equation}
For any $p\geq1$, the Malliavin derivative operator $D$ is closable from $L^p(\Omega)$ to $L^p(\Omega;H)$. Its domain is denoted by $\mathbb D^{1,p}$ with respect to the norm
\[\parallel F\parallel_{1,p}=\left(\mathbb E[\mid F\mid^p]+\mathbb E[\parallel DF\parallel_{H}^p]\right)^{\frac{1}{p}}. \]
We also introduce $\delta$, the Skorohod integral, defined as the adjoint operator of $D$ which is a linear operator on $L^2([0,T]\times\Omega,\mathbb R^d)$ with values in $L^2(\Omega)$ and we denote by $\text{Dom}(\delta)$ its domain.\\
We now state the basic properties of the Malliavin derivative and the Skorohod integral.
\begin{prop}[Chain rule property]\label{Cha} Fix $p\geq1$. For $\varphi\in C_b^1(\mathbb R^n,\mathbb R)$ and $F=(F_1,...,F_n)$ a random vector whose components belong to $\mathbb D^{1,p}$, $\varphi(F)\in\mathbb D^{1,p}$ and for $t\in[0,T]$, one has
\begin{equation}
\label{c}
D_t\varphi(F)= \sum_{i=1}^n\frac{\partial \varphi}{\partial x_i}(F)D_tF_i.
\end{equation}
\end{prop}
\begin{prop}
Let $\{X_t,~t\geq0\}$ be an $\mathbb R^n$ valued It$\hat{\text{o}}$ process whose dynamics are governed by the stochastic differential equation
\begin{equation}
dX_t=b(X_t)dt+\sigma(X_t)dW_t,
\end{equation}
where $b$ and $\sigma$ are supposed to be continuously differentiable functionals with bounded derivatives and $\sigma(x)\neq0$ for all $x\in\mathbb R^n$. Let $\{Y_t,~t\geq0\}$ be the associated first variation process given by the stochastic differential equation
\begin{equation}
dY_t=b^{\prime}(X_t)Y_tdt+\sum_{i=1}^n\sigma_i^{\prime}(X_t)Y_tdW_t^i,~~~~Y_0=I_n,
\end{equation}
where $I_n$ is the identity matrix of $\mathbb R^n$, primes denote derivatives and $\sigma_i$ is the $i$-th column vector of $\sigma$. The the process $\{X_t,~t\geq0\}$ belongs to $\mathbb D^{1,2}$ and its Malliavin derivative is given by
 \begin{equation}
D_rX_t=Y_tY_r^{-1}\sigma(X_r)1_{\{r\leq t\}}, ~~r\geq0~~a.s.,
\end{equation}
which is equivalent to
\begin{equation}
Y_t=D_rX_t\sigma^{-1}(X_r)Y_r1_{\{r\leq t\}}~~~~~~a.s.
\end{equation}
\end{prop}
\begin{prop}[Integration by parts formula]\label{Int}
If $u$ belongs to $\text{Dom}(\delta)$, then $\delta(u)=\int_0^Tu_t\delta W_t$ is the element of $L^2(\Omega)$ characterised by the integration by parts formula
\begin{equation}
\label{f}
\forall F\in\mathbb D^{1,2}~~~~\mathbb{E}[F\delta(u)]=\mathbb{E}\left[\int_0^TD_tF\cdot u(t)dt \right].
\end{equation}
\end{prop}
\begin{prop}\label{Int1}
If $u$ is an adapted process belonging to $L^2([0,T]\times\Omega,\mathbb R^d)$ then the Skorohod integral and the It$\hat{\text{o}}$ integral coincide, that is,
\begin{equation}
\delta(u)=\int_0^Tu(t)dW(t).
\end{equation}
\end{prop}
\begin{prop}\label{Skro}
If $F\in\mathbb D^{1,2}$ and $u\in\text{Dom}(\delta)$ such that $\mathbb E[F^2\int_0^Tu_t^2dt]<\infty$, one has
\begin{equation}
\label{j}
\delta(Fu)=F\delta(u)-\int_0^TD_tF\cdot u_tdt
\end{equation}
whenever the right hand side belongs to $L^2(\Omega)$. In particular, if $u$ is in addition adapted, one simply has
\begin{equation}
 \delta(Fu)=F\int_0^Tu_tdW_t-\int_0^TD_tF_t\cdot u_tdt.
\end{equation}
\end{prop}

\section{Asset Dynamics}
To compute price sensitivities of (\ref{fa3}) with respect to model parameters of the two underlying assets we must specify the risk-neutral dynamics of the two underlying assets $S_1(t)$ and $S_2(t)$. We consider the underlying assets $S_1(t)$ and $S_2(t)$ under the risk-neutral measure to be given by the two-dimensional system of It$\hat{\text{o}}$ stochastic differential equations of the form
\begin{eqnarray}\label{a}
dS_1(t)&=& S_1(t)[\mu_1(S_1(t),S_2(t))dt+\sigma_1(S_1(t),S_2(t))dW_1(t)]~~~S_1(0)=x_1,\nonumber\\
dS_2(t)&=& S_2(t)[\mu_2(S_1(t),S_2(t))dt+\sigma_2(S_1(t),S_2(t))dW_2(t)]~~~S_2(0)=x_2,
\end{eqnarray}
where $W_1(t)$ and $W_2(t)$ are correlated standard Brownian motions with correlation coefficient $\rho(-1,1)$. The coefficients $\mu_1$, $\mu_2$, $\sigma_1$ and $\sigma_2$ are assumed to satisfy the usual conditions to ensure the existence and uniqueness of a strong solution of (\ref{a}).\\
Given an arbitrary ${W}_1$, there exists $\widetilde{W}_2$ which is independent of ${W}_1$ and $W_2$. Then we can express $W_2$ as follows $W_2(t)=\rho {W}_1(t)+\sqrt{1 - \rho^2}\widetilde{W}_2(t)$. We also express ${W_1}$ as $W_1(t)=\widetilde{W}_1(t)$. We can rewrite (\ref{a}) as
\begin{eqnarray}\label{b}
dS_1(t)&=& S_1(t)[\mu_1(S_1(t),S_2(t))dt+\sigma_1(S_1(t),S_2(t))d\widetilde{W}_1(t)]\nonumber\\
dS_2(t)&=& S_2(t)[\mu_2(S_1(t),S_2(t))dt+\rho\sigma_2(S_1(t),S_2(t))d\widetilde{W}_1(t)+\sigma_2(S_1(t),S_2(t))\sqrt{1-\rho^2}d\widetilde{W}_2(t)
].\nonumber\\
\end{eqnarray}
Setting $S(t)=(S_1(t),S_2(t))^*$ and $\mathbb{W}=(W_1(t),W_2(t))^*$ ($(\cdot)^*$ denote the transpose of $(\cdot)$) and a two-dimensional notation we can write (\ref{b}) as
\begin{equation}\label{c}
dS(t)=\beta(S(t))dt+a(S(t))d\mathbb{W}(t)
\end{equation}
where
\[\beta(S(t))=\left(
  \begin{array}{c}
  \mu_1(S(t))S_1(t) \\
 \mu_2(S(t))S_2(t)
  \end{array}
\right)\]
and \[a(S(t))=\left(
  \begin{array}{cc}
    \sigma_1(S(t))S_1(t) & 0\\
    \rho\sigma_2(S(t))S_2(t) & \sqrt{1-\rho^2}\sigma_2(S(t))S_2(t)\\
  \end{array}
\right)\,.\]
We assume that $\beta$ and $a$ are both at least twice continuously differentiable functions with bounded derivatives and $a(x)\neq0$ for all $x\in\mathbb R^2$. To ensure that (\ref{c}) has a unique strong solution we further assume that $\beta$ and $a$ satisfy the Lipschitz and polynomial growth conditions.\\
The first variation process $\{Y(t),~0\leq t\leq T\}$ associated to $\{S(t),~0\leq t\leq T\}$ given in (\ref{c}) is defined by the stochastic differential equation
\begin{equation}\label{d}
dY(t)=\beta^{\prime}(S(t))Y(t)dt+\sum_{i=1}^2a_i^{\prime}(S(t))Y(t)dW_i(t), ~~~Y(0)=I_2,
\end{equation}
where $I_2$ is the $2\times2$ identity matrix of $\mathbb R^2$, the primes denote derivatives and $a_i$ is the $\text{i-th}$ column matrix of $a$.\\
\section{Computation of price sensitivities}
Following Fourni$\acute{\text{e}}$ et al. \cite{Fo}, we assume that the diffusion matrix $a$ satisfies the uniform elliptic condition:
\begin{equation}\label{e}
\exists \epsilon>0~~~~(a(x)\xi)^*(a(x)\xi)>\epsilon\mid\xi\mid^2~~~\text{for
all}~~x, \xi\in\mathbb R^n,~~\text{with}~~\xi\neq0. \end{equation}
In the computation of Greeks via Malliavin calculus, a weight function which is independent of the payoff function is obtained. To obtain a valid computation result, one has to guarantee that the Malliavin weights do not degenerate with probability one. To avoid this degeneracy we introduce the set $\Upsilon_n$ (see \cite{Fo}) defined by
\begin{equation}
\Upsilon_n=\{\alpha\in L^2([0,T])\mid\int_0^{t_i}\alpha(t)dt=1~~\text {for
all}~~ i=1,\ldots,n\}.
\end{equation}
We need the following lemma \cite{Mh}.
\begin{lem}
\label{12}
If $(Y(t)Y^{-1}(r)\alpha(r))\in L^2([0,T]\times\Omega)$ for all $r,t\in[0,T]$, then $S(t)$ is Malliavin differentiable and the Malliavin derivative of $X(t)$ can be written as follows:
\begin{equation}
D_rS(t)=Y(t)Y^{-1}(r)a(S(r))\textbf{1}_{r\leq t},~~~r\geq0,~~\text{a.s.}
\end{equation}
which is equivalent to
\begin{equation}
Y(t)=\int_0^TD_rS(t)\alpha(r)a^{-1}(S(r))Y(r)dr~~~~~\forall \alpha\in \Upsilon_n.
\end{equation}
\end{lem}
Let $\Phi:\mathbb R^2\rightarrow\mathbb R$ be a measurable function. We assume that the $\Phi$ satisfies the integrability condition
\[\mathbb E[\Phi^2(S_1(T),S_2(T)]<\infty. \]
From the arbitrage theory, the price of the spread option can be expressed in terms of the expectation as in (\ref{fa3}). We have the following results.

\begin{prop}\label{La}
Suppose that the functions $\beta$ and $a$ in (\ref{c}) are continuously differentiable with bounded derivatives, and the diffusion matrix $a$ satisfies the uniform ellipticity condition (\ref{e}). In addition, the spread payoff function $\Phi$ is square integrable and continuously differentiable with bounded derivatives. Then for all $\alpha\in \Upsilon_n$, we have
\begin{equation}
\Delta_1=\mathbb E[e^{-rT}\Phi(S_1(T),S_2(T))\pi^{\Delta_1}],
\end{equation}
where the Malliavin weight $\pi^{\Delta_1}$ is
\begin{equation}\label{mi}
\pi^{\Delta_1}=\int_0^T\alpha(t)(a^{-1}(S(t))Y(t))^*d\mathbb W(t).
\end{equation}
\end{prop}
\begin{proof}
First assume $\Phi\in C_b^2(\mathbb R^2;\mathbb R)$ We have
\begin{eqnarray*}
\Delta_1&=&\frac{\partial}{\partial x_1}\mathbb E[e^{-rT}\Phi(S_1(T),S_2(T))]=\mathbb E[e^{-rT}\frac{\partial}{\partial x_1}\Phi(S_1(T),S_2(T))]\\
&=&\mathbb E[e^{-rT}\Phi^{\prime}(S_1(T),S_2(T))\frac{\partial S_1(T)}{\partial x_1}]=\mathbb E[e^{-rT}\Phi^{\prime}(S_1(T),S_2(T))Y_1(T)],
\end{eqnarray*}
where the interchange of the derivative and the expectation is justified by the dominated convergence theorem. In fact, as $\varepsilon\rightarrow0$
\[ \frac{\Phi(S_1(T)(1+\frac{\varepsilon}{x_1}),S_2(T))-\Phi(S_1(T),S_2(T))}{\varepsilon}\rightarrow\Phi^{\prime}(S_1(T),S_2(T))Y_1(T)~~~\text{a.s}\]
and by the Taylor theorem
\begin{eqnarray*}
&&\left|\frac{\Phi(S_1(T)(1+\frac{\varepsilon}{x_1}),S_2(T))-\Phi(S_1(T),S_2(T))}{\varepsilon}\right|\\
&&\leq\int_0^1\mathbb E\left[\left|\Phi^{\prime}(S_1(T)(1+\frac{\delta\varepsilon}{x_1}),S_2(T))\left(\frac{S_1(T)(1+\frac{\delta\varepsilon}{x_1})}{x_1}\right)\right|\right]d\delta
\end{eqnarray*}
which is clearly uniformly bounded in $\varepsilon$. This proves that
\[\Delta_1=\mathbb E[e^{-rT}\Phi^{\prime}(S_1(T),S_2(T))Y_1(T)].\]
From Lemma \ref{12} we have
\[\Delta_1=\mathbb E\left[\int_0^Te^{-rT}\Phi^{\prime}(S_1(T),S_2(T))D_rS_1(t)\alpha(t)a^{-1}(S(t))Y_1(t)dt\right].\]
An application of Proposition \ref{Cha} and Proposition \ref{Int}, and the fact that the Skorohod integral coincides with the It$\hat{\text{o}}$ stochastic integral (Proposition \ref{Int1}) yields
\begin{eqnarray*}
\Delta_1&=&\mathbb E\left[\int_0^Te^{-rT}D_s\Phi(S_1(T),S_2(T))\alpha(t)a^{-1}(S(t))Y_1(t)dt\right]\\
&=&\mathbb E\left[e^{-rT}\Phi(S_1(T),S_2(T))\int_0^T\alpha(t)\left(a^{-1}(S(t))Y(t) \right)^*dW(t)\right].
\end{eqnarray*}
This is the desired result for $\Phi\in C_b^2(\mathbb R^2;\mathbb R)$. It remains to remove the regularity assumption on $\Phi$. To this end, we consider $\Phi$ such that $\mathbb E[e^{-rT}\Phi^2(S_1(T),S_2(T))$ is locally uniformly bounded in $x_1$. We can always find a sequence $\{\Phi_n\}_{n\in\mathbb N}$ of contentiously differentiable functions from $\mathbb R^2$ to $\mathbb R$ with compact support such that
\[ \lim_{n\rightarrow+\infty}\mathbb E[|e^{-rT}\Phi_n(S_1(T),S_2(T))-e^{-rT}\Phi(S_1(T),S_2(T))|]=0.\]
Hence by the Cauchy-Schwartz inequality with $\mathbb E[|\pi^{\Delta_1}|^2]<\infty$ we have that for each $x_1$
\begin{eqnarray*}
&&\left|\mathbb E\left[e^{-rT}\Phi_n(S_1(T),S_2(T))\pi^{\Delta_1} - e^{-rT}\Phi(S_1(T),S_2(T))\pi^{\Delta_1}\right] \right|^2\\
&&\leq
\mathbb E\left[\left|e^{-rT}\Phi_n(S_1(T),S_2(T))-e^{-rT}\Phi(S_1(T),S_2(T))\right|^2\right]\mathbb E\left[\left|\pi^{\Delta_1}\right|^2\right]\\
&&\rightarrow0  \end{eqnarray*}
as $n\rightarrow\infty$. Then we obtain that
\begin{eqnarray}\label{Se}
&&\mathbb E\left[e^{-rT}\Phi_n(S_1(T)(1+\frac{\varepsilon}{x_1}),S_2(T))\right]-\mathbb E\left[e^{-rT}\Phi_n(S_1(T),S_2(T))\right]\nonumber\\
&&=\int_0^{\varepsilon}
\mathbb E\left[e^{-rT}\Phi_n(S_1(T)(1+\frac{h}{x_1}),S_2(T))\pi^{\Delta_1}\right]dh \end{eqnarray}
Then, by taking limits as $n$ tends to $\infty$, we obtain that $\mathbb E[e^{-rT}\Phi(S_1(T)(1+\frac{\varepsilon}{x_1}),S_2(T))]$ is continuous in $x_1$. In a similar fashion, we can prove that $\mathbb E[e^{-rT}\Phi(S_1(T),S_2(T))\pi^{\Delta_1}]$ is continuous in $x_1$. Finally, by taking limits as $n$ tends to $\infty$ in (\ref{Se}) and dividing by $\varepsilon$, we obtain that $\mathbb E[e^{-rT}\Phi(S_1(T),S_2(T))]$ is differentiable with respect to $x_1$ and the desired formula holds. This completes the proof.
\end{proof}
\begin{prop}\label{La1}
Suppose that the functions $\beta$ and $a$ in (\ref{c}) are continuously differentiable with bounded derivatives, and the diffusion matrix $a$ satisfies the uniform ellipticity condition (\ref{e}). In addition, the spread payoff function $\Phi$ is square integrable and continuously differentiable with bounded derivatives. Then for all $\alpha\in \Upsilon_n$, we have
\begin{equation}
\Delta_2=\mathbb E[e^{-rT}\Phi(S_1(T),S_2(T))\pi^{\Delta_2}],
\end{equation}
where the Malliavin weight $\pi^{\Delta_2}$ is
\begin{equation}
\pi^{\Delta_2}=\int_0^T\alpha(t)(a^{-1}(S(t))Y(t))^*d\mathbb W(t).
\end{equation}
\end{prop}
\begin{proof} The proof follows the same arguments as in the proof of Proposition \ref{La}.
\end{proof}

\begin{prop}\label{far}
Suppose that the functions $\beta$ and $a$ in (\ref{c}) are continuously differentiable with bounded derivatives, and the diffusion matrix $a$ satisfies the uniform ellipticity condition (\ref{e}). In addition, the spread payoff function $\Phi$ is square integrable and continuously differentiable with bounded derivatives. Then for all $\alpha\in \Upsilon_n$, we have
\begin{equation}
\Gamma_1=\mathbb E[e^{-rT}\Phi(S_1(T),S_2(T))\pi^{\Gamma_1}],
\end{equation}
where the Malliavin weight $\pi^{\Gamma_1}$ is
\begin{equation}\label{mi1}
\pi^{\Gamma_1}=\left(\pi^{\Delta_1}\right)^2-\frac{1}{x_1}\pi^{\Delta_1}-\int_0^T\alpha(t)\left(a^{-1}(S(t))Y_1(t)\right)^2dt.
\end{equation}
\end{prop}
\begin{proof}
As in the proof of Proposition \ref{La}, it suffices to show the result with $\Phi\in C_b^2(\mathbb R^2;\mathbb R)$. Define $G:=x_1\pi^{\Delta_1}$ so that $\pi^{\Delta_1}=\frac{1}{x_1}G$. Then $\frac{\partial \pi^{\Delta_1}}{\partial x_1}=-\frac{1}{x_1^2}G=-\frac{1}{x_1}\pi^{\Delta_1}$. We have

\begin{eqnarray}\label{kl}
 \Gamma_1&=&\frac{\partial^2}{\partial x_1^2}\mathbb E[e^{-rT}\Phi(S_1(T),S_2(T))]\nonumber\\
  &=& \frac{\partial}{\partial x_1}\mathbb E[e^{-rT}\Phi(S_1(T),S_2(T))\pi^{\Delta_1}]\nonumber\\
  &=&\mathbb E[e^{-rT}\Phi^{\prime}(S_1(T),S_2(T))Y_1(T)\pi^{\Delta_1}] -\frac{1}{x_1}\mathbb E[e^{-rT}\Phi(S_1(T),S_2(T))\pi^{\Delta_1}].
 \end{eqnarray}
where the interchange of the derivative and the expectation are justified by the dominated convergence theorem. For the first term in (\ref{kl}) we use similar arguments as in the proof of Proposition \ref{La}:
\begin{eqnarray*}
&&\mathbb E[e^{-rT}\Phi^{\prime}(S_1(T),S_2(T))Y_1(T)\pi^{\Delta_1}]\nonumber\\&&=
\mathbb E\left[\int_0^Te^{-rT}D_t\Phi(S_1(T),S_2(T))\alpha(t)(a^{-1}(S(t))Y_1(t))\pi^{\Delta_1}\right]\\
&&=\mathbb E\left[e^{-rT}\Phi(S_1(T),S_2(T))\delta\left(\alpha(\cdot)(a^{-1}(X(\cdot))Y_1(\cdot))\pi^{\Delta_1}\right)\right].\end{eqnarray*}
Finally, applying the the integration by parts formula (Proposition \ref{Skro}) with $D_t\pi^{\Delta_1}=\alpha(t)(a^{-1}(S(t))Y(t))$, we have
\begin{eqnarray}\label{Sk1}
\delta\left(\alpha(\cdot)(a^{-1}(X(\cdot))Y_1(\cdot))\pi^{\Delta_1}\right)&=&\pi^{\Delta_1}\int_0^T\alpha(t) (a^{-1}(S(t))Y_1(t))dW_1(t)\nonumber\\&&-\int_0^T\alpha(t)\left(a^{-1}(S(t))Y_1(t)\right)^2dt\nonumber\\
&=&\left(\pi^{\Delta_1}\right)^2-\int_0^T\alpha(t)\left(a^{-1}(S(t))Y_1(t)\right)^2dt.
\end{eqnarray}
Combining (\ref{kl}) with (\ref{Sk1}) we get the desired result for $\Phi\in C_b^2(\mathbb R^2;\mathbb R)$. To remove the regularity assumption on $\Phi$ we need
\[\mathbb E\left[\left|e^{-rT}\Phi(S_1(T),S_2(T))\left(\pi^{\Delta_1}\right)^4\right|\right]^2\leq \mathbb E[e^{-2rT}\Phi^2(S_1(T),S_2(T))]\mathbb E[\left(\pi^{\Delta_1}\right)^8]<+\infty. \]
which complete the proof.
\end{proof}
\begin{prop}\label{far1}
Suppose that the functions $\beta$ and $a$ in (\ref{c}) are continuously differentiable with bounded derivatives, and the diffusion matrix $a$ satisfies the uniform ellipticity condition (\ref{e}). In addition, the spread payoff function $\Phi$ is square integrable and continuously differentiable with bounded derivatives. Then for all $\alpha\in \Upsilon_n$, we have
\begin{equation}
\Gamma_2=\mathbb E[e^{-rT}\Phi(S_1(T),S_2(T))\pi^{\Gamma_2}],
\end{equation}
where the Malliavin weight $\pi^{\Gamma_1}$ is
\begin{equation}
\pi^{\Gamma_2}=\left(\pi^{\Delta_2}\right)^2-\frac{1}{x_2}\pi^{\Delta_2}-\int_0^T\alpha(t)\left(a^{-1}(S(t))Y_2(t)\right)^2dt.
\end{equation}
\end{prop}
\begin{proof}
The proof follows the same arguments as in the proof of Proposition \ref{far}.
\end{proof}

 Next we consider the derivative of the price of the spread option with respect to
  the volatilities $\sigma_i$, $i=1,2$. This computation is not as straightforward as in the computations of $\Delta_i$ and $\Gamma_i$, for $i=1,2$. Here, instead of computing the derivative of the price of the spread option with respect to the associated volatilities, we consider adding a perturbation term $\varepsilon$ to the volatility term and then observe the effect of the perturbation on the price of the spread option. To avoid degeneracy, we introduce the set of deterministic function
  \[\widetilde{\Upsilon}_n=\{ \widetilde{\alpha}\in L^2([0,T]):\int_{t_{i-1}}^{t_i}\widetilde{\alpha}(t)dt,~~~\forall i=1,...,n.\} \]
  Let $\gamma:\mathbb R^+\times\mathbb R^2\rightarrow\mathbb R^2\times\mathbb R^2$ a direction function for the volatility term such that $\varepsilon\in[-1,1]$, $\gamma$ and $\sigma+\varepsilon\gamma$ are continuously differentiable with bounded derivatives and verify Lipschitz conditions such that the following uniform elliptic condition is satisfied:
  \begin{equation}\label{Ass1}
  \exists\epsilon>0~~~\xi^*(a+\varepsilon\gamma)^*(x)(a+\varepsilon\gamma)(x)\xi\geq\epsilon\parallel\xi\parallel^2,~~~\forall~~\xi,x\in\mathbb R^2, ~~\xi\neq0.
  \end{equation}
  As in Fourni$\acute{\text{e}}$ et al. \cite{Fo}, we consider the perturbed process $(S^{\varepsilon}(t))_{t\in[0,T]}$ as a solution of the following stochastic differential equation
  \begin{equation}\label{ju}
dS^{\varepsilon}(t)=\beta(S^{\varepsilon}(t))dt+[a(S^{\varepsilon}(t))+\varepsilon\gamma(S^{\varepsilon}(t))]d\mathbb W(t), ~~~S^{\varepsilon}(0)=x.
 \end{equation}
 WE also relate to this perturbed process the perturbed price of the spread option $u^{\varepsilon}$ defined by
 \begin{equation}
 u^{\varepsilon}=\mathbb E[e^{-rT}\Phi(S_1^{\varepsilon}(T),S_2^{\varepsilon}(T))].
 \end{equation}
 We also introduce a variation process with respect to $\varepsilon$, which is the derivative of $X^{\varepsilon}(t)$ with respect to the parameter $\varepsilon$ ($Z^{\varepsilon}(t):=\frac{\partial X^{\varepsilon}(t)}{\partial\varepsilon}$):

\begin{equation}\label{po}
dZ^{\varepsilon}(t)=\beta^{\prime}(S^{\varepsilon}(t))Z^{\varepsilon}(t)dt+
\sum_{i=1}^2(a_i^{\prime}(S^{\varepsilon}(t))+ \varepsilon\gamma_i^{\prime}(S^{\varepsilon}(t)))Z^{\varepsilon}(t)dW_i(t)+\gamma(S^{\varepsilon}(t))d\mathbb W(t),~~Z^{\varepsilon}(0)=0_{2\times 2}\nonumber\\\end{equation}
where $0_{2\times2}$ is the zero column vector in $\mathbb R^2$ and $\gamma_i^{\prime}$ denotes the derivative of the i-th column.

\begin{prop}\label{zzz3} Assume that the uniformly elliptic condition (\ref{Ass1}) holds and for $B(t_i)=Y^{-1}(t_i)Z(t_i)=Y^{-1}(t_i)Z^{\varepsilon=0}(t_i), ~i=1,2,$ there exists $a^{-1}(X)YB\in{\text{Dom}}(\delta)$. Then, for any square integrable spread option payoff function, $\Phi$, with continuously differentiable and bounded derivatives,
 \begin{equation}
\frac{\partial}{\partial
\varepsilon}u^{\varepsilon}\mid_{\varepsilon=0}=\mathbb
E[e^{-rT}\Phi(S_1(T),S_2(T))\delta\left(a^{-1}(S(\cdot))Y(\cdot)\widetilde{B}(\cdot)\right)]
\end{equation}
holds. Here
\[\widetilde{B}(t)=\sum_{i=1}^2\widetilde{\alpha}(t)(B({t_i})-B({t_{i-1}}))1_{\{t\in[t_{i-1},t_i]\}}, \]
 for $t_0=0$ and $\widetilde{\alpha}\in\tilde{\Upsilon_n}$. Moreover, if $B$ is Malliavin differentiable, then
 \begin{eqnarray*}
 \delta\left(a^{-1}(S(\cdot)Y(\cdot)\widetilde{B}(\cdot) \right)&=&\sum_{i=1}^2\left\{B^*(t_i)\int_{t_{i-1}}^{t_i}\widetilde{\alpha}(t)(a^{-1}(S(t))Y(t))^*d\mathbb W(t)\right.\\
 &&\left.-\int_{t_{i-1}}^{t_i}\widetilde{\alpha}(t)\text{Tr}((D_tB(t_i))a^{-1}(S(t))Y(t))dt\right.\\
 &&\left.- \int_{t_{i-1}}^{t_i}\widetilde{\alpha}(t)(a^{-1}(S(t))Y(t)B(t_{i-1}))^*d\mathbb W(t)\right\}.
 \end{eqnarray*}
\end{prop}
\begin{proof}
The proof follow the same line of argument as in the proof of Proposition 3.1.5 in \cite{Mh}.
\end{proof}

\section{Examples}
We consider the risk-neutral price dynamics given by the following systems of stochastic differential equations
\begin{eqnarray}\label{jul}
dS_1(t)&=&S_1(t)[(r-q_1)dt+\sigma_1dW_1(t)],~~~S_1(0)=x_1,\nonumber\\
dS_2(t)&=&S_2(t)[(r-q_2)dt+\sigma_2dW_2(t)],~~~S_2(0)=x_2,
\end{eqnarray}
where $q_1$ and $q_2$ are the instantaneous dividend yields, $\sigma_1$ and $\sigma_2$ are positive constants volatilities, $S_1(t)$ and $S_2(t)$ are prices of two assets at time $t$, and $W_1(t)$ and $W_2(t)$ are two standard Brownian motions with correlation parameter $\rho\in(-1,1).$ For all $t\in[0,T]$ we define
\[ \widetilde{W}_2(t):=\frac{1}{\sqrt{1-\rho^2}}\left(W_2(t)-\rho{W}_1(t)\right) ~~~\text{and}~~~W_1(t)=\widetilde{W}_1(t).\]
The process $\{\widetilde{W}_2(t)~~0\leq t\leq T\}$ is a Brownian motion which is independent of $W_1(t)$ and ${W}_2(t)$. Then the system of stochastic differential equations (\ref{jul}) can be rewritten in matrix form
\[dS(t)=\beta(S(t))dt+a(S(t))d\mathbb W(t),~~~X(0)=(x_1,x_2)\]
where
\[\beta(S(t))=\left(
  \begin{array}{c}
   (r-q_1)S_1(t) \\
 (r-q_2)S_2(t)
  \end{array}
\right)\]
and
\[a(S(t))=\left(
  \begin{array}{cc}
    \sigma_1 S_1(t) & 0 \\
    \rho\sigma_2S_2(t) & \sigma_2\sqrt{1-\rho^2}S_2(t)\\
  \end{array}
\right)\,.\]
The inverse of $a$ is
\[a^{-1}(S(t))=\frac{1}{\sigma_1\sigma_2\sqrt{1-\rho^2}S_1(t)S_2(t)}\left(
  \begin{array}{cc}
    \sigma_2\sqrt{1-\rho^2}S_2(t) & 0\\
    -\rho\sigma_2S_2(t) & \sigma_1 S_1(t)\\
  \end{array}
\right)\,.\]
The first variation process is given by
\[dY(t)=\beta^{\prime}(S(t))Y(t)dt+a_1^{\prime}(S(t))Y(t)dW_1(t)+a_2^{\prime}(S(t))Y(t)dW_2(t),~~~Y(0)=I,\]
where
\[\beta^{\prime}(S(t))=\left(
  \begin{array}{cc}
    r-q_1 & 0\\
    0 & r-q_2\\
  \end{array}
\right)\,,~~~~
a_1^{\prime}(S(t))=\left(
  \begin{array}{cc}
    \sigma_1& 0\\
    0 & \rho\sigma_2\\
  \end{array}
\right)\,,
~~\text{and}~~a_2^{\prime}(S(t))=\left(
  \begin{array}{cc}
    0 & 0\\
    0 & \sigma_2\sqrt{1-\rho^2}\\
  \end{array}
\right)\,.\]
The matrix $(a^{-1}(S(t))Y(t))^*$ has the following form
\[(a^{-1}(S(t))Y(t))^* =\left(
  \begin{array}{cc}
    \frac{Y^{11}(t)}{\sigma_1S_1(t)} & -\frac{\rho Y^{11}(t)}{\sigma_1\sqrt{1-\rho^2}S_1(t)}\\
    \frac{Y^{12}(t)}{\sigma_1 S_1(t)} & -\frac{\rho Y^{12}(t)}{\sigma_1S_1(t)\sqrt{1-\rho^2}}+\frac{Y^{22}(t)}{\sigma_2S_2(t)\sqrt{1-\rho^2}}\\
  \end{array}
\right)\,.\]
An application of Proposition \ref{La} with $Y^{11}(t)=\frac{S_1(t)}{x_1}$, $Y^{22}(t)=\frac{S_2(t)}{x_2}$ and $Y^{21}(t)=0$ yields
\[\Delta_1=\mathbb E\left[e^{-rT}\Phi(S_1(T),S_2(T))\frac{\sqrt{1-\rho^2} W_1(T)-\rho W_2(T)}{\sigma_1x_1T\sqrt{1-\rho^2}}\right]\] and an application of Proposition \ref{La1} yields
\[\Delta_2=\mathbb E\left[e^{-rT}\Phi(S_1(T),S_2(T))\frac{W_2(T)}{\sigma_2x_2T\sqrt{1-\rho^2}}\right].\]
For Gamma, an application of Propositions \ref{far} - \ref{far1} yields
\[ \Gamma_1=\mathbb E\left[e^{-rT}\Phi(S_1(T),S_2(T))\left\{ \left(\pi^{\Delta_1}\right)^2-\frac{1}{x_1}\left(\pi^{\Delta_1}\right)-\frac{1}{Tx_1^2\sigma_1^2(1-\rho^2)} \right\}\right] \]
and
\[\Gamma_2=\mathbb E\left[e^{-rT}\Phi(S_1(T),S_2(T))\left\{ \left(\pi^{\Delta_2}\right)^2-\frac{1}{x_2}\left(\pi^{\Delta_2}\right)-\frac{1}{Tx_2^2\sigma_2^2(1-\rho^2)} \right\} \right] .\]
For \emph{Vega} we perturbed diffusion matrix of the equation for $S_1(t)$ with $\gamma$. We note from (\ref{po}) that
\begin{equation}\label{po1}dZ_1(t)=(r-q_1)Z_1(t)dt+\sigma_1Z_1(t)dW_1(t)+S_1(t)dW_1(t)\end{equation} where we have chosen $\gamma$ to be
\[\gamma(S(t))=\left(
  \begin{array}{cc}
    S_1(t)&0\\
    0&0\\
      \end{array}
\right)\,.\]
Since $S_2(t)$ do not depend on $S_1(t)$ we set $Z_2(t)=0$. An application of the It$\hat{\text{o}}$ formula to (\ref{po1}) yields the following solution
\[ Z_1(t)=S_1(t)(W_1(t)-\sigma_1t). \]
Since $B(t)=Y^{-1}(t)Z(t)$, we have
\[B(T)=x_1(W_1(T)-\sigma_1T).\]
The Malliavin derivative of $B(T)$ is calculated as follows
\[ D_tB(T)=D_t(x_1(W_1(T)-\sigma_1T))=x_1. \]
 An application of Proposition \ref{zzz3} yields
\[\mathcal V_1=\mathbb E[e^{-rT}\Phi(S_T,S_2(T))\pi^{\mathcal V_1}]  \]
where the Malliavin weight $\pi^{\mathcal V_1}$ is given by
\[\pi^{\mathcal V_1}=\frac{1}{T}\left\{\left(W_1(T)-\sigma_1T\right)\left(\frac{W_1(T)}{\sigma_1}-\frac{\rho W_2(T)}{\sigma_1\sqrt{1-\rho^2}}\right)-\frac{T}{\sigma_1}\right\}.\]
Following a similar procedure and applying Proposition \ref{zzz3} we obtain
\[\mathcal V_2=\mathbb E[e^{-rT}\Phi(S_T,S_2(T))\pi^{\mathcal V_2}]  \]
where the Malliavin weight $\pi^{\mathcal V_2}$ is given by
\newpage
\begin{eqnarray*}\pi^{\mathcal V_2}&=&\frac{1}{T}\left\{\left(W_2(T)-\rho\sigma_2T-\sigma_2\sqrt{1-\rho^2}T\right)\right.\\
&&\left.\times\left(\int_0^T\frac{Y^{12}(t)dW_1(t)}{\sigma_1S_1(t)}-
\frac{\rho }{\sigma_2\sqrt{1-\rho^2}}\int_0^T\frac{Y^{12}(t)dW_2(t)}{S_1(t)}+\frac{W_2(T)}{\sigma_2\sqrt{1-\rho^2}}\right)\right.\\
&&\left.-\frac{T}{\sigma_2\sqrt{1-\rho^2}}\right\}.\end{eqnarray*}

\section{Application of spread options with stochastic volatility }
The stochastic volatility has an important impact on the price of the spread option. The stochastic volatility helps to understand the market evolution completely. We set up the dynamics with stochastic volatility as follows
\begin{eqnarray}\label{Fa5}
dS_1(t)&=&S_1(t)[\mu_1(t)dt+\sigma_1\sqrt{V(t)}dW_1(t)],~~~S_1(0)=x_1,\nonumber\\
dS_2(t)&=&S_2(t)[\mu_2(t)dt+\sigma_2\sqrt{V(t)}dW_2(t)],~~~S_2(0)=x_2,\nonumber\\
dV(t)&=&\kappa(1-V(t))dt+\nu\sqrt{V(t)}dZ(t),~~~~V(0)=v_0,
\end{eqnarray}
where $S_i(t)$, $i=1,2$ denote the asset prices and $V(t)$ represents a volatility factor, $\kappa$ measures the speed at which $V(t)$ reverts towards $1$, $\nu$ is the parameter which determines the volatility of the variance process.\\
We specify that $dW_1(t)dW_2(t)=\rho dt$ and $dW_i(t)dZ(t)=0$, $i=1,2.$\\
Due to the independence of $W_i$, $i=1,2$ and $Z$, we have that $S_i(T)$, $i=1,2$ given by the integrated variance
\[ \overline{V}(T):=\int_0^TV(t)dt,\]
is lognormally distributed.This result can be generalised to two-dimensional case where $S_1(T)$ and $S_2(T)$ given $\overline{V}(T)$ are jointly lognormal. This argument, however, is not pursued in this paper.\\
Define
\[ W_2(t):=\rho\widetilde{W}_1(t)+\sqrt{1-\rho^2}\widetilde{W}_2(t),~~W_1(t)=\widetilde{W}_1(t), ~~\text{and}~~Z(t)=\widetilde{W}_3(t).\]
The system of stochastic differential equation (\ref{Fa5}) can be written in matrix form
\begin{equation}\label{Fa7}
dS(t)=\beta(S(t))dt+a(S(t))d\mathbb W(t)
\end{equation}
where $S(t)=(S_1(t),S_2(t),V(t))^*$, $\mathbb W(t)=(W_1(t),W_2(t),Z(t))^*$,

\[\beta(S(t))=\left(
  \begin{array}{c}
    \mu(t)S_1(t)\\
    \mu_2(t)S_2(t)\\
    \kappa(1-V(t))
  \end{array}
\right)\,,\]
and
\[a(S(t))=\left(
  \begin{array}{ccc}
    \sigma_1\sqrt{V(t)} S_1(t) & 0 & 0\\
    \rho\sigma_2\sqrt{V(t)}S_2(t) & \sigma_2\sqrt{1-\rho^2}\sqrt{V(t)}S_2(t) & 0\\
    0& 0 & \nu\sqrt{V(t)}
  \end{array}
\right)\,. \]
The inverse of $a$ is
\[a^{-1}(S(t))=\left(
  \begin{array}{ccc}
   \frac{1}{\sigma_1\sqrt{V(t)} S_1(t)} & 0 & 0\\
   -\frac{\rho}{\sigma_1\sqrt{1-\rho^2}\sqrt{V(t)}S_1(t)} & \frac{1}{\sigma_2\sqrt{1-\rho^2}\sqrt{V(t)}S_2(t)}& 0\\
    0& 0& \frac{1}{\nu\sqrt{V(t)}}
  \end{array}
\right)\,.\]
The first variation process is given by
\[dY(t)=\beta^{\prime}(S(t))Y(t)dt+a_1^{\prime}(S(t))Y(t)dW_1(t)+a_2^{\prime}(S(t))Y(t)dW_2(t)+a_3^{\prime}(S(t))Y(t)dW_3(t)\]
where
\[\beta^{\prime}(S(t))=\left(
  \begin{array}{ccc}
    \mu_1(t) & 0&0\\
    0 & \mu_2(t)&0\\
    0&0&-\kappa
  \end{array}
\right)\,,~~~a_1^{\prime}(S(t))=\left(
  \begin{array}{ccc}
    \sigma_1\sqrt{V(t)} & 0  & \frac{\sigma_1 S_1(t)}{2\sqrt{V(t)}}\\
    0 & \rho\sigma_2\sqrt{V(t)}& \frac{\rho\sigma_2S_2(t)}{2\sqrt{V(t)}}\\
    0& 0& 0
  \end{array}
\right)\,,\]

\[a_2^{\prime}(S(t))=\left(
  \begin{array}{ccc}
    0 & 0  & 0\\
    0 & \sigma_2\sqrt{1-\rho^2}\sqrt{V(t)}& \frac{\sigma_2 \sqrt{1-\rho^2}S_2(t)}{2\sqrt{V(t)}}\\
    0& 0& 0
  \end{array}
\right)\,,\]
and
\[a_3^{\prime}(S(t))=\left(
  \begin{array}{ccc}
   0 & 0  & 0\\
    0 & 0&0\\
    0& 0& \frac{\nu}{2\sqrt{V(t)}}
  \end{array}
\right)\,.\]
The matrix $(a^{-1}(S(t))Y(t))^*$ has the following form
\[(a^{-1}(S(t))Y(t))^*=\left(
  \begin{array}{ccc}
   \frac{Y^{11}(t)}{\sigma_1\sqrt{V(t)}S_1(t)} & -\frac{\rho Y^{11}(t)}{\sigma_1\sqrt{1-\rho^2}\sqrt{V(t)}S_1(t)}  & 0\\
    \frac{Y^{12}(t)}{\sigma_1\sqrt{V(t)}S_1(t)}& -\frac{\rho Y^{12}(t)}{\sigma_1\sqrt{1-\rho^2}\sqrt{V(t)}S_1(t)}+\frac{Y^{22}(t)}{\sigma_2\sqrt{1-\rho^2}\sqrt{V(t)}S_2(t)}&0\\
    \frac{Y^{13}(t)}{\sigma_1\sqrt{V(t)}S_1(t)}& -\frac{\rho Y^{13}(t)}{\sigma_1\sqrt{1-\rho^2}\sqrt{V(t)}S_1(t)}& \frac{Y^{33}(t)}{\nu\sqrt{V(t)}}
  \end{array}
\right)\,.\]
This, with the applications of Propositions \ref{La}  - \ref{far1}, yields the following results.
\begin{prop}
Suppose that $\beta$ and $a$ in (\ref{Fa7}) are continuously differentiable functions with bounded derivatives and that diffusion matrix $a$ satisfies the uniform ellipticity condition (\ref{e}). Then, for any $\alpha\in\Upsilon_n$ and $\Phi(S_1(T),S_2(T))$ square integrable, we have \begin{enumerate}
  \item \[\Delta_1=\mathbb E\left[e^{-rT}\Phi(S_1(T),S_2(T))\pi^{\Delta_1}\right]\]
  where $\pi^{\Delta_1}$ is the Malliavin weight given by
  \begin{equation}\label{de} \pi^{\Delta_1}=\frac{1}{\sigma_1x_1T}\int_0^T\frac{1}{\sqrt{V(t)}}dW_1(t)-
  \frac{\rho}{\sigma_1\sqrt{1-\rho^2}x_1T}\int_0^T\frac{1}{\sqrt{V(t)}}dW_2(t).\end{equation}
  \item \[\Delta_2=\mathbb E\left[e^{-rT}\Phi(S_1(T),S_2(T))\pi^{\Delta_2}\right]\]
  where $\pi^{\Delta_2}$ is the Malliavin weight given by
  \begin{eqnarray}\label{de2} \pi^{\Delta_2}&=&\frac{1}{\sigma_1 T}\int_0^T\frac{Y^{12}(t)}{\sqrt{V(t)}S_1(t)}dW_1(t)-\frac{\rho}{\sigma_1\sqrt{1-\rho^2}T}\int_0^T\frac{Y^{12}(t)}{\sqrt{V(t)}S_1(t)}dW_2(t)
  \nonumber\\ &&+\frac{1}{\sigma_2\sqrt{1-\rho^2}x_2T}\int_0^T\frac{1}{\sqrt{V(t)}}dW_2(t).\end{eqnarray}

\end{enumerate}
\end{prop}
\begin{prop}
Suppose that $\beta$ and $a$ in (\ref{Fa7}) are continuously differentiable functions with bounded derivatives and that diffusion matrix $a$ satisfies the uniform ellipticity condition (\ref{e}). Then, for any $\alpha\in\Upsilon_n$ and $\Phi(S_1(T),S_2(T))$ square integrable, we have \begin{enumerate}
  \item \[\Gamma_1=\mathbb E\left[e^{-rT}\Phi(S_1(T),S_2(T))\pi^{\Gamma_1}\right]\]
  where $\pi^{\Gamma_1}$ is the Malliavin weight given by
  \[ \pi^{\Gamma_1}= \left(\pi^{\Delta_1}\right)^2-\frac{1}{x_1}\pi^{\Delta_1}-\frac{1}{T\sigma_1^2x_1^2(1-\rho^2)}\int_0^T\frac{1}{V(t)}dt,\]
  where $\pi^{\Delta_1}$ is given in (\ref{de}).
  \item \[\Gamma_2=\mathbb E\left[e^{-rT}\Phi(S_1(T),S_2(T))\pi^{\Gamma_2}\right]\]
  where $\pi^{\Gamma_2}$ is the Malliavin weight given by
  \[\pi^{\Gamma_2}=\left(\pi^{\Delta_2}\right)^2-\frac{1}{x_1}\pi^{\Delta_2}-\frac{1}{T\sigma_2^2x_2^2(1-\rho^2)}\int_0^T\frac{1}{V(t)}dt,\]
  where $\pi^{\Delta_2}$ is given in (\ref{de2}).
\end{enumerate}
\end{prop}

To evaluate $\mathcal V_1$ we consider the perturbed process given by (\ref{ju}) where $\gamma$ is chosen to be
\[\gamma(S(t))=\left(
  \begin{array}{ccc}
   S_1(t) & 0  & 0\\
    0 & 0&0\\
    0& 0& 0
  \end{array}
\right)\,,\]
where the perturbation is only on the asset price $S_1(t)$. Since $S_2(t)$ and $V(t)$ do not depend on $\epsilon$, we deduce that $Z_1(t)=0$ and $Z_3(t)=0,$ respectively. From (\ref{po}) we deduce that $Z_1(t)$ satisfies the following stochastic differential equation
\[ dZ_1(t)=\mu_1Z_1(t)dt+\sigma_1\sqrt{V(t)}Z_1(t)dW_1(t)+S_1(t)dW_1(t).\]
An application of It$\hat{\text{o}}$ formula yield the solution
\[ Z_1(t)=S_1(t)\left(W_1(t)-\int_0^t\sigma_1\sqrt{V(s)}ds\right).\]
Since $B(t)=Y^{-1}(t)Z(t)$, one has
\[ B(t)=x_1\left(W_1(t)-\int_0^t\sigma_1\sqrt{V(s)}ds\right). \]
An application of Proposition \ref{Cha} and Lemma \ref{12} gives
\begin{eqnarray*} D_tB(T)&=&x_1\left((1,0,0)^*-\sigma_1\int_0^TD_t\sqrt{V(s)}ds\right)\\&=&x_1\left((1,0,0)^*-
\frac{\sigma_1}{2}\int_t^T\frac{\sqrt{V(t)}}{\sqrt{V(s)}}\frac{Y^{22}(s)}{Y^{22}(t)}\left(\rho,\sqrt{1-\rho^2},0\right)^*ds.\right)\end{eqnarray*}
Then, one obtains
\[ \text{Tr}\left((D_tB(T))a^{-1}(S(t))Y(t)\right)=\frac{1}{\sigma_1\sqrt{V(t)}}. \]
The application of Proposition \ref{zzz3} yields the following result.
\begin{prop}\label{ha}
Suppose that $\beta$ and $a$ in (\ref{Fa7}) are continuously differentiable functions with bounded derivatives and that diffusion matrix $a$ satisfies the uniform ellipticity condition (\ref{Ass1}). Then, for any $\Phi(S_1(T),S_2(T))$ square integrable with $\widetilde{\alpha}(t)=\frac{1}{T}$, we have \[\mathcal V_1=\mathbb E[e^{-rT}\Phi(S_1(T),S_2(T))\pi^{\mathcal V_1}]  \]
where the Malliavin weight $\pi^{\mathcal V_1}$ is given by
\begin{eqnarray*}
\pi^{\mathcal V_1}&=&\frac{1}{T}\left\{\left(W_1(T)-\int_0^T\sigma_1\sqrt{V(t)}dt\right)\left(\int_0^T\frac{dW_1(t)}{\sigma_1\sqrt{V(t)}}-\frac{\rho }{\sigma_1\sqrt{1-\rho^2}}\int_0^T\frac{dW_2(t)}{\sqrt{V(t)}}\right)\right.\\
&&\left.-\int_0^T\frac{dt}{\sigma_1\sqrt{V(t)}}\right\}.
\end{eqnarray*}
\end{prop}

A similar procedure used to obtain Proposition \ref{ha} with perturbation only in the diffusion coefficient of second asset, $S_1(t)$, together with the application of Proposition \ref{zzz3} yields the following result.
\begin{prop}
Suppose that $\beta$ and $a$ in (\ref{Fa7}) are continuously differentiable functions with bounded derivatives and that diffusion matrix $a$ satisfies the uniform ellipticity condition (\ref{Ass1}). Then, for any $\Phi(S_1(T),S_2(T))$ square integrable with $\widetilde{\alpha}(t)=\frac{1}{T}$, we have \[\mathcal V_2=\mathbb E[e^{-rT}\Phi(S_1(T),S_2(T))\pi^{\mathcal V_2}]  \]
where the Malliavin weight $\pi^{\mathcal V_2}$ is given by
\begin{eqnarray*}
\pi^{\mathcal V_2}&=&\frac{1}{T}\left\{\left(W_2(T)-\int_0^T\rho\sigma_2\sqrt{V(t)}dt-\int_0^T\sigma_2\sqrt{1-\rho^2}\sqrt{V(t)}dt\right)\right.\\
&&\left. \times\left(\int_0^T\frac{Y^{12}(t)dW_1(t)}{\sigma_1\sqrt{V(t)}S_1(t)}-\frac{\rho }{\sigma_1\sqrt{1-\rho^2}}\int_0^T\frac{Y^{12}(t)dW_2(t)}{\sqrt{V(t)}S_1(t)}+
\frac{1}{\sigma_2\sqrt{1-\rho^2}}\int_0^T\frac{dW_2(t)}{\sqrt{V(t)}}\right)\right.\\
&&\left.-\int_0^T\frac{dt}{\sigma_2\sqrt{1-\rho^2}\sqrt{V(t)}}\right\}.
\end{eqnarray*}
\end{prop}

\section{Localised Malliavin approach for spread options}
Following Fourni$\acute{\text{e}}$ et al. \cite{Fo}, the variance reduction is achieved by using a localized Malliavin technique. The approach is to localise the Malliavin weights round the strike price $K$, that is, instead of using the Malliavin calculus approach to derive the Greeks globally, the calculus is only applied locally around the singularities of the payoff function. In our case we localise the payoff function $\Phi(s_1,s_2)$ around $s_2-s_1=K$. Direct methods can be used outside the localisation point.\\
To be precise we introduce the Lipchitz continuous approximation to the Heaviside function:
\begin{displaymath}
H_a(s_1,s_2)=\left\{\begin{array}{lll} & 0\quad\quad\quad\quad\quad\quad\quad~\text{if}~~s_2-s_1<K-a,\\
&  \frac{s_2-s_1-(K-a)}{2a}\quad\text{if}~~K-a\leq s_2-s_1\leq K+a,\\
& 1\quad\quad\quad\quad\quad\quad\text{if}~~s_2-s_1>K+a.
\end{array}\right.
\end{displaymath}
for a constant $a$. We set
\begin{displaymath}
h_a(s_1,s_2)=\int_{-\infty}^{s1\wedge s_2}H_a(y_1,y_2)dy=\left\{\begin{array}{lll} & 0\quad\quad\quad\quad\quad\quad\quad\text{if}~~s_2-s_1<K-a,\\
&  \frac{(s_2-s_1-(K-a))^2}{4a}\quad\text{if}~~K-a\leq s_2-s_1\leq K+a,\\
& s_2-s_1-K\quad\quad\text{if}~~~s_2-s_1>K+a.
\end{array}\right.
\end{displaymath}
We observe that $h_a^{\prime}(s_1,s_2)=H_a(s_1,s_2)$. \\
Set
\[ \Phi_a(s_1,s_2)=\Phi(s_1,s_2)-h_a(s_1,s_2) = (s_2-s_1-K)^+-h_a(s_1,s_2).\]
Notice that $\Phi_a(s_1,s_2)$ vanishes for $s_2-s_1<K-a$ and for $s_2-s_1\geq K+a$. This means that $\Phi_a(s_1,s_2)$ is a localised version of $\Phi(s_1,s_2)$.\\ We can, therefore, write
\[\Phi(S_1(T),S_2(T))=\Phi_a(S_1(T),S_2(T))+h_a(S_1(T),S_2(T)),\]
so that the price of the spread call option is given by
\[ u=\mathbb E[e^{-rT}\Phi_a(S_1(T),S_2(T))]+\mathbb E[e^{-rT}h_a(S_1(T),S_2(T))].\]
We illustrate how the localised Malliavin calculus approach is applied to derive the expression for $\Delta_1$ and $\Gamma_1$. We have
\begin{eqnarray*}
\Delta_1&=&\frac{\partial}{\partial x_1}\mathbb E[e^{-rT}\Phi_a(S_1(T),S_2(T))]+\frac{\partial}{\partial x_1}\mathbb E[e^{-rT}h_a(S_1(T),S_2(T))]\\
&=& \mathbb E[e^{-rT}\Phi_a(S_1(T),S_2(T))\pi^{\Delta_1}]+\mathbb E[e^{-rT}H_a(S_1(T),S_2(T))Y_1(T)],
\end{eqnarray*}
with $\pi^{\Delta_1}$ given by (\ref{mi}), where the second equality is due to the application of  Proposition \ref{La} on the first term and direct differentiation on the second term. A similar procedure can be applied to obtain $\Delta_2$.\\
For $\Gamma_1$, we have
\begin{eqnarray*}
\Gamma_1&=&\frac{\partial^2}{\partial x_1^2}\mathbb E[e^{-rT}\Phi_a(S_1(T),S_2(T))]+\frac{\partial^2}{\partial x_1^2}\mathbb E[e^{-rT}h_a(S_1(T),S_2(T))]\\
&=& \mathbb E[e^{-rT}\Phi_a(S_1(T),S_2(T))\pi^{\Gamma_1}]+\frac{1}{x_1^2}\mathbb E[e^{-rT}H_a^{\prime}(S_1(T),S_2(T)) S_1^2(T)],
\end{eqnarray*}
with $\pi^{\Gamma_1}$ given by (\ref{mi1}), where the second equality is due to the application of Proposition \ref{far} on the first term and direct differentiation on the second term. A similar procedure can be applied to obtain $\Gamma_2$.\\

\section{Concluding remarks}
In this paper, tractable formulae for price sensitivities of spread options are presented using Malliavin calculus approach. Precisely, the formulae are presented for asset dynamics driven by geometric Brownian motion and stochastic volatility models of the financial markets. In order to apply the Malliavin calculus approach, the joint characteristic function of the underlying assets is not required. In addition, the Malliavin calculus approach avoids the direct differentiation of the payoff functional. Computing price sensitivities of spread options via the tractable formulae obtained guarantees a convergence rate that is independent of the regularity of the payoff function and dimensionality \cite{Ban}. The results obtained in this paper form a generalization of the computation of price sensitivities for spread options. The price sensitivities are expressed in terms of the expectation of spread option payoff multiplied with some Malliavin weights which are functions of the underlying asset expressed as stochastic integrals. This is in agreement with results in Fourni$\acute{\text{e}}$ et al. \cite{Fo}. The Malliavin weight functions are independent of the payoff functional, this is suitable for Monte Carlo methods which can be applied for general spread options, and not specifically for each spread option. The use of localised Malliavin approach help to reduce the variance when the Monte Carlo methods are applied by localising the Malliavin calculus around the point od discontinuity. Price sensitivities are valuable to investors as well as financial institutions  as they are used to find and construct financial risk strategies to hedge against potential sources of the underlying price risk. It remains to future research to consider spread options driven by general L\'evy models. It would also be interesting to consider the effect of stochastic correlation in the model.

\subsection*{Acknowledgment}
This work was supported in part by the National Research Foundation of South Africa (Grant Number: 105924).

 \end{document}